\documentclass[12pt]{article}
\usepackage{amsmath,amsfonts,amssymb,amsthm,latexsym}
\newtheorem{thm}{Theorem}
\newtheorem{defn}[thm]{Definition}
\newtheorem{lem}[thm]{Lemma}
\newtheorem{cor}[thm]{Corollary}
\newtheorem{exam}[thm]{Example}
\newtheorem{prop}[thm]{Proposition}
\newtheorem{rem}[thm]{Remark}

\DeclareMathOperator{\fix}{Fix}
\begin{document}

\centerline {\Large\textbf{On metric-preserving functions and fixed point theorems}}
\vskip.8cm \centerline {\textbf{Prapanpong Pongsriiam}$^a$\footnote{ Corresponding author} and \textbf{Imchit Termwuttipong}$^b$}

\vskip.5cm
\centerline {$^a$Department of Mathematics, Faculty of Science}
\centerline {Silpakorn University, Nakhon Pathom, 73000, Thailand}
\centerline {{\tt prapanpong@gmail.com}}
\centerline {$^b$Department of Mathematics and Computer Science, Faculty of Science}
\centerline {Chulalongkorn University, Bangkok, 10330, Thailand}
\centerline {{\tt  Imchit.T@chula.ac.th}}

\begin{abstract}
Kirk and Shahzad have recently given fixed point theorems concerning local radial contractions and metric transforms. In this article, we replace the metric transforms by metric-preserving functions. This in turn gives several extensions of the main results given by Kirk and Shahzad. Several examples are given. The fixed point sets of metric transforms and metric-preserving functions are also investigated.  
\end{abstract}
\noindent \textbf{Keywords:} metric-preserving function; metric transform; local radial contraction; rectifiably pathwise connected space; uniform local multivalued contraction 
\section{Introduction}
\indent The concept of metric transforms is introduced by L. M. Blumenthal \cite{B1, B2} in 1936 while the concept of metric-preserving functions seems to be introduced by W. A. Wilson \cite{Wi} in 1935 and is investigated in details by many authors \cite{BD, BD1, C, Da, Do2, Do, Do1, DP, DP1, DP2, PRS, PV, Po, Po1, Sr, TO, V}. Recently, Petru\c{s}el, Rus, and \c{S}erban \cite{PRS} have shown the role of metric-preserving functions in fixed point theory. In addition, Kirk and Shahzad \cite{KS} have given results concerning metric transforms and fixed point theorems. Their main results are as follows:
\begin{thm}\label{thmA}
(Kirk and Shahzad \cite[Theorem 2.2]{KS}) Let $(X,d)$ be a metric space and $g:X\to X$. Suppose there exists a metric transform $\phi$ on $X$ and a number $k\in(0,1)$ such that the following conditions hold:
\begin{itemize}
\item[(a)] For each $x\in X$ there exists $\varepsilon_x > 0$ such that for every $u\in X$
$$
d(x,u)<\varepsilon \Rightarrow (\phi\circ d)(g(x),g(u)) \leq kd(x,u).
$$
\item[(b)] There exists $c\in (0,1)$ such that for all $t>0$ sufficiently small 
$$
kt\leq \phi(ct).
$$
\end{itemize}
Then $g$ is a local radial contraction on $(X,d)$.
\end{thm}
\begin{thm}\label{thmB}
(Kirk and Shahzad \cite[Theorem 2.3]{KS}) Suppose, in addition to the assumptions in Theorem \ref{thmA}, $X$ is complete and rectifiably pathwise connected. Then $g$ has a unique fixed point $x_0$, and $\lim_{n\to\infty}g^n(x) = x_0$ for each $x\in X$. 
\end{thm}
Our purpose is to show that the metric transform $\phi$ in Theorem \ref{thmA} can be replaced by a metric-preserving function. This in turn gives extensions to the main results given by Kirk and Shahzad in \cite[Theorem 2.2, Theorem 2.3, Theorem 2.8, Theorem 3.4, and Theorem 3.6]{KS}. Now let us recall some basic definitions that will be used throughout this article.
\begin{defn}
Let $f:[0,\infty)\to [0,\infty)$. Then 
\begin{itemize}
\item[(i)] $f$ is said to be \emph{a metric transform} if $f(0) = 0$, $f$ is strictly increasing on $[0,\infty)$, and $f$ is concave on $[0,\infty)$,
\item[(ii)] $f$ is said to be \emph{a metric-preserving function} if for all metric spaces $(X,d)$, $f\circ d$ is a metric on $X$,
\item[(iii)] $f$ is said to be \emph{amenable} if $f^{-1}(\{0\}) = \{0\}$,
\item[(iv)] $f$ is said to be \emph{tightly bounded} if there exists $u>0$ such that $f(x)\in [u,2u]$ for all $x>0$,
\item[(v)]  $f$ is said to be \emph{subadditive} if $f(a+b)\leq f(a)+f(b)$ for all $a, b\in [0,\infty)$.
\end{itemize}
\end{defn}
\begin{defn}
Let $(X,d)$ be a metric space and $g:X\to X$. Then $g$ is said to be \emph{a local radial contraction} if there exists $k\in (0,1)$ such that for each $x\in X$, there exists $\varepsilon > 0$ such that for every $u\in X$,
$$
d(x,u)<\varepsilon \quad\Rightarrow\quad d(g(x),g(u)) \leq kd(x,u). 
$$
\end{defn}
\begin{defn}
Let $(X,d)$ be a metric space and $\gamma$ be a path in $X$, that is, a continuous map $\gamma:[a,b]\to X$. A \emph{partition} $Y$ of $[a,b]$ is a finite collection of points $Y = \{y_0, \ldots, y_N\}$ such that $a = y_0\leq y_1\leq y_2\leq\cdots\leq y_N = b$. The supremum of the sums
$$
\sum Y = \sum_{i=1}^Nd(\gamma(y_{i-1}),\gamma(y_i))
$$
over all the partitions $Y$ of $[a,b]$ is called the \emph{length} of $\gamma$. A path is said to be \emph{rectifiable} if its length is finite. A metric space is said to be \emph{rectifiably pathwise connected} if each two points of $X$ can be joined by a rectifiable path. 
\end{defn}
We will give some auxiliary results in Section 2. Then we will give the results concerning metric-preserving functions, local radial contractions, and uniform local multivalued contractions in Section 3 and Section 4. Finally, we investigate the fixed point sets of metric transforms and metric-preserving functions in Section 5.
\section{Lemmas}
We need to use some properties of metric-preserving functions and some fixed point theorems. We give them in this section for the convenience of the reader. For more details of the metric-preserving functions, we refer the reader to \cite{C, Do2, Do1}.
\begin{lem}\label{lemma3}
Let $f:[0,\infty)\to[0,\infty)$. Then 
\begin{itemize}
\item[(i)] if $f$ is metric-preserving, then $f$ is amenable,
\item[(ii)] if $f$ is amenable and concave, then $f$ is metric-preserving.
\end{itemize}
\end{lem}
\begin{proof}
The proof of (i) is easily obtained, see for example, in \cite[Lemma 2.3]{BD1}. The proof of (ii) is given in \cite[Proposition 1.2]{BD1} and \cite[p.\ 13]{Do2}. See also \cite[Proposition 2]{BD} and \cite[p.\ 311]{C}.
\end{proof}
\begin{lem}\label{chapter2lemma9}
Let $f:[0,\infty)\to[0,\infty)$. If $f$ is amenable, subadditive, and increasing, then $f$ is metric-preserving.
\end{lem}
\begin{proof}
The proof can be found in \cite[Proposition 1.1]{BD1}, \cite[Proposition 2.3]{C}, and \cite[p.\ 9]{Do2}.
\end{proof}
\begin{lem}\label{lemma8}
If $f:[0,\infty)\to[0,\infty)$ is amenable and tightly bounded, then $f$ is metric-preserving.
\end{lem}
\begin{proof}
The proof is given in \cite[Proposition 1.3]{BD1}, \cite[Proposition 2.8]{C}, and \cite[p.\ 17]{Do2}.
\end{proof}
\begin{lem}\label{lemma10}
If $f$ is metric-preserving and $0\leq a\leq 2b$, then $f(a) \leq 2f(b)$.
\end{lem}
\begin{proof}
The proof is given in \cite[Lemma 2.5]{BD1}, and \cite[p.\ 16]{Do2}.
\end{proof}
For a metric-preserving function $f$, let $K_f$ denote the set 
$$
K_f = \left\{k>0\mid f(x)\leq kx\quad\text{for all $x\geq 0$}\right\}.
$$
Recall also that we define $\inf \emptyset = +\infty$. Then we have the following result.
\begin{lem}\label{lemma5}
Let $f:[0,\infty)\to[0,\infty)$ be metric-preserving. Then $f'(0) = \inf K_f$. In particular, $f'(0)$ always exist in $\mathbb R\cup \{+\infty\}$ and
\begin{itemize}
\item[(i)] $f'(0) < +\infty$ if and only if $K_f\neq \emptyset$, and
\item[(ii)] $f'(0) = +\infty$ if and only if $K_f = \emptyset$.
\end{itemize}
\end{lem}
\begin{proof}
The proof can be found in \cite[Theorem 2]{BD}, \cite[Theorem 4.4]{C}, and \cite[p.\ 37--39]{Do2}.
\end{proof}
The next lemma is probably well-known but we give a proof here for completeness.
\begin{lem}\label{amencavegotdecreas}
If $f:[0,\infty)\to[0,\infty)$ is amenable and concave, then the function $x\mapsto \frac{f(x)}{x}$ is decreasing on $(0,\infty)$
\end{lem}
\begin{proof}
Let $a, b\in (0,\infty)$ and $a<b$. Since $f$ is concave, we obtain
$$
f(a) = f\left(\left(1-\frac{a}{b}\right)(0)+\left(\frac{a}{b}\right)(b)\right)\geq \left(1-\frac{a}{b}\right)f(0)+\frac{a}{b}f(b) = \frac{a}{b}f(b). 
$$
Therefore $\frac{f(a)}{a} \geq \frac{f(b)}{b}$, as desired.
\end{proof}
\begin{lem}\label{Pokornylemma}
(Pokorn\'y \cite{Po}) Let $f:[0,\infty)\to[0,\infty)$. Assume that $f$ is amenable and there is a periodic function $g$ such that $f(x) = x+g(x)$ for all $x\geq 0$. Then $f$ is metric-preserving if and only if $f$ is increasing and subadditive.
\end{lem}
\begin{proof}
The proof can be found in \cite[p.\ 32]{Do2} and \cite[Theorem 1]{Po}.
\end{proof}
\begin{lem}\label{lemma13}
(Hu and Kirk \cite{HK}) Let $(X,d)$ be a complete metric space for which each two points can be joined by a rectifiable path, and suppose $g:X\to X$ is a local radial contraction. Then $g$ has a unique fixed point $x_0$, and $\lim_{n\to\infty}g^n(x) = x_0$ for each $x\in X$.
\end{lem}
As noted by Kirk and Shahzad \cite{KS}, an assertion in the proof of Lemma \ref{lemma13} given in \cite{HK} was based on a false proposition of Holmes \cite{H}. But Jungck \cite{J} proved that the assertion itself is true. Hence the proof given in \cite{HK} with minor changes is true. Kirk and Shahzad \cite{KS} apply Tan's result \cite{T} to extend some of their theorems. We will also apply Tan's result as well.
\begin{lem}\label{lemma12}
(Tan \cite{T}) Let $X$ be a topological space, let $x_0\in X$, and let $g:X\to X$ be a mapping for which $f:=g^N$ satisfies $\lim_{n\to\infty}f^n(x) = x_0$ for each $x\in X$. Then $\lim_{n\to\infty}g^n(x) = x_0$ for each $x\in X$. (Also if $x_0$ is the unique fixed point of $f$, it is also the unique fixed point of $g$.)
\end{lem} 
We will use Nadler's result concerning set-valued mappings. So let us recall some more definitions. If $\varepsilon>0$ is given, a metric space $(X,d)$ is said to be \emph{$\varepsilon$-chainable} if given $a, b\in X$ there exist $x_1, x_2, \ldots, x_n\in X$ such that $a = x_1$, $b = x_n$, and $d(x_i,x_{i+1})<\varepsilon$ for all $i\in \{1, 2, \ldots, n-1\}$. The result of Nadler that we need is the following.
\begin{lem}\label{nadlerthm}
(Nadler \cite{Na}) Let $(X,d)$ be a complete $\varepsilon$-chainable metric space. If $T:X\to \mathcal{CB}(X)$ is an $(\varepsilon,k)$-uniform local multivalued contraction, then $T$ has a fixed point.
\end{lem}
%
%
%
%
%
%
%
\section{Local radial contractions and metric-preserving functions}
In this section, we will give a generalization of Theorem \ref{thmA} where the metric transform $\phi$ is replaced by a metric-preserving function. In fact, we obtain a more general result as follows:
\begin{thm}\label{thm2.1}
Let $(X,d)$ be a metric space and let $g:X\to X$. Assume that there exists $k\in (0,1)$ and a metric-preserving function $f$ satisfying the following conditions:
\begin{itemize}
\item[(a)] for each $x\in X$, there exists $\varepsilon>0$ such that for every $u\in X$ 
$$
d(x,u)<\varepsilon \quad\Rightarrow\quad (f\circ d)(g(x),g(u)) \leq kd(x,u),\quad\text{and}
$$
\item[(b)] $f'(0) > k$.
\end{itemize} 
Then $g$ is a local radial contraction.
\end{thm}
We know from Lemma \ref{lemma5} that $f'(0)$ always exists in $\mathbb R\cup\{+\infty\}$. So condition (b) in Theorem \ref{thm2.1} makes sense. To prove this theorem, we will first show that $g$ is continuous in the following lemma.
\begin{lem}\label{lemma2.2}
Suppose that the assumptions in Theorem \ref{thm2.1} hold. Then the function $g$ is continuous.
\end{lem}
 As a consequence of Theorem \ref{thm2.1}, we can replace the metric transform $\phi$ in Theorem \ref{thmA} by a metric-preserving function and obtain an extension of Theorem \ref{thmA}.
\begin{thm}\label{thm2.5}
With the same assumptions in Theorem \ref{thm2.1} except that condition (b) is replaced by (b$'$): there exists $c\in (0,1)$ such that $f(ct)\geq kt$ for all $t>0$ sufficiently small. Then $g$ is a local radial contraction.
\end{thm}
\begin{rem}
As noted by Kirk and Shahzad \cite[Remark 2.5]{KS}, \cite[Proposition 2.6]{KS}, metric transforms satisfying condition (b) in Theorem \ref{thmA} are numerous. Proposition \ref{prop2.3}, Example \ref{examafterprop2.3}, and Example \ref{example25fab} (to be given after the proof of Theorem \ref{thm2.5}) show that the class of metric-preserving functions satisfying condition (b) in Theorem \ref{thmA} is larger than the class of metric transforms satisfying the same condition. Hence the class of such functions is even more numerous and Theorem \ref{thm2.5} is indeed an extension of Theorem \ref{thmA}.
\end{rem}
Now let us give the proof of Lemma \ref{lemma2.2}, Theorem \ref{thm2.1}, and Theorem \ref{thm2.5} as follows.\vspace{0.3cm}\\
\noindent\textbf{Proof of Lemma \ref{lemma2.2}}\vspace{0.2cm}\\
Let $x\in X$ and let $\varepsilon >0$. Since $k<f'(0) = \lim_{y\to 0^+}\frac{f(y)-f(0)}{y-0} = \lim_{y\to 0^+}\frac{f(y)}{y}$, there exists $\delta_1 > 0$ such that 
\begin{equation}\label{eq1lemma2.2}
0<y\leq \delta_1 \Rightarrow \frac{f(y)}{y} > k.
\end{equation}
By condition (a), there exists $\delta_2 > 0$ such that for every $u\in X$,
\begin{equation}\label{eq2lemma2.2}
d(x,u) < \delta_2 \Rightarrow (f\circ d)(g(x),g(u)) \leq kd(x,u).
\end{equation} 
Let $\delta_3 = \min\{\delta_1, \delta_2, \varepsilon\}$. Then by (\ref{eq1lemma2.2}), we obtain 
\begin{equation}\label{eq3lemma2.2}
\frac{f(\delta_3)}{\delta_3} > k.
\end{equation}
Since $f$ is metric-preserving, we obtain by Lemma \ref{lemma10}, and (\ref{eq3lemma2.2}) that for every $b\in [0,\infty)$
\begin{equation}\label{eq4lemma2.2}
b\geq \frac{\delta_3}{2} \Rightarrow f(b) \geq \frac{f(\delta_3)}{2} > \frac{k\delta_3}{2}.
\end{equation}
Now let $\delta = \frac{\delta_3}{2}$ and $u\in X$ be such that $d(x,u)<\delta$. Then by (\ref{eq2lemma2.2}), we obtain
$$
f(d(g(x),g(u))) \leq kd(x,u) < k\delta = \frac{k\delta_3}{2}.
$$
Then by (\ref{eq4lemma2.2}), $d(g(x),g(u)) < \frac{\delta_3}{2} \leq \frac{\varepsilon}{2} < \varepsilon$. This shows that $g$ is continuous, as required. \qed\vspace{0.3cm}\\
\noindent\textbf{Proof of Theorem \ref{thm2.1}}\vspace{0.2cm}\\
Let $c = \frac{1}{2}\left(\frac{k}{f'(0)}+1\right)$ where if $f'(0) = +\infty$, we define $\frac{k}{f'(0)}$ to be zero and $c = \frac{1}{2}(0+1) = \frac{1}{2}$. Then $0 \leq \frac{k}{f'(0)} < c < 1$. Consider 
$$
f'(0) = \lim_{y\to 0^+}\frac{f(y)-f(0)}{y-0} = \lim_{y\to 0^+}\frac{f(y)}{y}.
$$
Since $f'(0) > \frac{k}{c}$, there exists $\delta_1>0$ such that 
\begin{equation}\label{eq1thm2.1}
0<y<\delta_1 \Rightarrow \frac{f(y)}{y} > \frac{k}{c}.
\end{equation}
To show that $g$ is a local radial contraction with the contraction constant $c$, let $x\in X$. By Lemma \ref{lemma2.2}, $g$ is continuous at $x$. So there exists $\delta_2 > 0$ such that for every $u\in X$, 
\begin{equation}\label{eq2thm2.1}
d(x,u)<\delta_2 \Rightarrow d(g(x),g(u)) < \delta_1.
\end{equation}
By condition (a), there exists $\delta_3>0$ such that for every $u\in X$,
\begin{equation}\label{eq3thm2.1}
d(x,u)<\delta_3 \Rightarrow (f\circ d)(g(x),g(u)) \leq kd(x,u).
\end{equation}
Now let $\varepsilon = \min\{\delta_1, \delta_2, \delta_3\}$ and $u\in X$ be such that $d(x,u)<\varepsilon$. We need to show that $d(g(x),g(u)) \leq cd(x,u)$. If $d(g(x),g(u)) = 0$, then we are done. So assume that $d(g(x),g(u)) > 0$. Then $0 < d(x,u) < \varepsilon$ and we obtain by (\ref{eq3thm2.1}) that 
\begin{equation}\label{eq4thm2.1}
\frac{(f\circ d)(g(x),g(u))}{d(x,u)} \leq k.
\end{equation}
The left hand side of (\ref{eq4thm2.1}) is 
\begin{align}\label{eq5thm2.1}
\frac{(f\circ d)(g(x),g(u))}{d(x,u)} &= \frac{f(d(g(x),g(u)))}{d(g(x),g(u))}\cdot \frac{d(g(x),g(u))}{d(x,u)}\notag\\
&> \frac{k}{c}\frac{d(g(x),g(u))}{d(x,u)},
\end{align}
where the above inequality is obtained from (\ref{eq2thm2.1}) and (\ref{eq1thm2.1}). From (\ref{eq4thm2.1}) and (\ref{eq5thm2.1}), we obtain 
$$
\frac{k}{c}\frac{d(g(x),g(u))}{d(x,u)} < k,
$$
which implies the desired result. This completes the proof. \qed\vspace{0.3cm}\\
\noindent\textbf{Proof of Theorem \ref{thm2.5}}\vspace{0.2cm}\\
By Lemma \ref{lemma5}, we know that $f'(0)$ exists in $\mathbb R\cup\{+\infty\}$ and by Theorem \ref{thm2.1}, it suffices to show that $f'(0)>k$. So we can assume further that $f'(0)$ exists in $\mathbb R$. Now $f'(0) = \lim_{y\to 0^+}\frac{f(y)-f(0)}{y-0} = \lim_{y\to 0^+}\frac{f(y)}{y}$. Since the limits involved in the following calculation exist, we obtain 
$$
\lim_{y\to 0^+}\frac{f(y)}{y} = \lim_{t\to 0^+}\frac{f(ct)}{ct} \geq \lim_{t\to 0^+}\frac{kt}{ct} = \frac{k}{c} > k. 
$$
Therefore $f'(0)>k$, as desired. \qed\vspace{0.3cm}\\
As noted earlier, we will show that the class of metric-preserving functions and the class of metric-preserving functions satisfying condition (b) in Theorem \ref{thmA} are, respectively, larger than the class of metric transforms and the class of metric transforms satisfying condition (b) in Theorem \ref{thmA}.
\begin{prop}\label{prop2.3}
Every metric transform is metric-preserving.
\end{prop}
\begin{proof}
Let $f$ be a metric transform. Since $f(0) = 0$ and $f$ is strictly increasing, $f$ is amenable. Since $f$ is amenable and concave, we obtain by Lemma \ref{lemma3} (ii) that $f$ is metric-preserving.
\end{proof}
\begin{cor}
Kirk and Shahzad's result (Theorem \ref{thmA}) holds.
\end{cor}
\begin{proof}
This follows immediately from Proposition \ref{prop2.3} and Theorem \ref{thm2.5}.
\end{proof}
\begin{exam}\label{examafterprop2.3}
Let $f, g, h :[0,\infty)\to [0,\infty)$ be given by 
\begin{align*}
f(x) &= \begin{cases}
0, &\text{if $x=0$};\\
1, &\text{if $x>0$ and $x\in \mathbb Q$};\\
2, &\text{if $x\in \mathbb Q^c$},
\end{cases} \quad\quad
g(x) = \begin{cases}
x, &\text{if $x\in[0,1]$};\\
1, &\text{if $x>1$},
\end{cases} \\
h(x) &= \begin{cases}
x, &x\in [0,1];\\
1, &x\in [1,10];\\
x-9, &x\in (10,11);\\
2, &x\geq 11.
\end{cases}
\end{align*}
Since $f(x)\in[1,2]$ for all $x>0$, $f$ is tightly bounded. Therefore by Lemma \ref{lemma8}, $f$ is metric-preserving. It is easy to see that $f$ is not increasing (and is not concave either). So $f$ is not a metric transform. It is easy to see that $g$ is amenable and concave, so it is metric-preserving, by Lemma \ref{lemma3} (ii). In addition, if $c = k = \frac{1}{2}\in(0,1)$, then $g(ct) \geq kt$ for all $t\in [0,1]$. So $g$ satisfies condition (b) in Theorem \ref{thmA}. But $g$ is not a metric transform because it is not strictly increasing. For $h$, we proved in \cite[Example 14]{PT1} that $h$ is metric-preserving. Similar to $g$, the function $h$ satisfies the condition (b) in Theorem \ref{thmA}. It is easy to see that $h$ is neither strictly increasing nor concave. Therefore $h$ is not a metric transform.
\end{exam}
We can generate more functions similar to $g$ given in Example \ref{examafterprop2.3} as follows.
\begin{exam}\label{example25fab}
Let $a\geq 1$ and $b>0$. Define $f_{a,b}:[0,\infty)\to[0,\infty)$ by 
$$
f_{a,b}(x) = \begin{cases}
ax, &\text{if $x\in[0,b]$};\\
ab, &\text{if $x>b$}.
\end{cases} 
$$
Then $f_{a,b}$ is amenable and concave. So by Lemma \ref{lemma3} (ii), $f_{a,b}$ is metric-preserving. We also have $f'_{a,b}(0) = a\geq 1$. So it satisfies condition (b) in Theorem \ref{thm2.1}. However, $f_{a,b}$ is not a metric transform because it is not strictly increasing.
\end{exam} 
\begin{rem}
Some natural questions concerning the relation of metric transforms, metric-preserving functions, and condition (b) can be answered by Example \ref{examafterprop2.3} and Example \ref{example25fab}:
\begin{itemize}
\item[Q1:] Is there a continuous metric-preserving function which is not a metric transform?
\item[A1:] Yes, $g$ and $h$ given in Example \ref{examafterprop2.3} and $f_{a,b}$ given in Example \ref{example25fab} are such functions.
\item[Q2:] Is there any nowhere continuous metric-preserving function which is not a metric transform?
\item[A2:] Yes, $f$ given in Example \ref{examafterprop2.3} is such a function.
\item[Q3:] Is there a nowhere monotone metric-preserving function which is not a metric transform?
\item[A3:] Yes, $f$ given in Example \ref{examafterprop2.3} is such a function.
\item[Q4:] Is there a  metric-preserving function which is concave and satisfies condition (b) in Theorem \ref{thmA} but it is not a metric transform?
\item[A4:] Yes, $g$ given in Example \ref{examafterprop2.3} and $f_{a,b}$ given in Example \ref{example25fab} are such functions.
\end{itemize}
\end{rem}
Now that we have obtained two extensions of Theorem \ref{thmA}, we give two generalizations of Theorem \ref{thmB} as follows.
\begin{thm}
The following statements hold:
\begin{itemize}
\item[(a)] Suppose, in addition to the assumptions in Theorem \ref{thm2.1}, $X$ is complete and rectifiably pathwise connected. Then $g$ has a unique fixed point $x_0$, and $\lim_{n\to\infty}g^n(x) = x_0$ for each $x\in X$.
\item[(b)] Suppose, in addition to the assumptions in Theorem \ref{thm2.5}, $X$ is complete and rectifiably pathwise connected. Then $g$ has a unique fixed point $x_0$, and $\lim_{n\to\infty}g^n(x) = x_0$ for each $x\in X$.
\end{itemize}
\end{thm}
\begin{proof}
Part (a) follows immediately from Theorem \ref{thm2.1} and Lemma \ref{lemma13}. Part (b) follows immediately from Theorem \ref{thm2.5} and Lemma \ref{lemma13}.
\end{proof}
Finally, we remark that Kirk and Shahzad use Tan's result (Lemma \ref{lemma12}) to extend Theorem \ref{thmB} further \cite[Theorem 2.3 and Theorem 2.8]{KS}. We similarly apply their argument to obtain the following.
\begin{thm}\label{thm26}
Let $X$ be a metric space which is complete and rectifiably pathwise connected, and suppose $g:X\to X$ is a mapping for which
\begin{itemize}
\item[(a)] $g^N$ satisfies the assumptions in Theorem \ref{thm2.1} for some $N\in \mathbb N$, or
\item[(b)] $g^M$ satisfies the assumptions in Theorem \ref{thm2.5} for some $M\in \mathbb N$.
\end{itemize}
Then $g$ has a unique fixed point $x_0$, and $\lim_{n\to\infty}g^n(x) = x_0$ for each $x\in X$.
\end{thm}
\begin{proof}
This follows immediately from Theorem \ref{thm2.1}, Theorem \ref{thm2.5}, Lemma \ref{lemma13}, and Lemma \ref{lemma12}.
\end{proof}
\noindent \textbf{Conclusion:} We have obtained extensions of the main results given by Kirk and Shahzad in \cite[Theorem 2.2, Theorem 2.3, and Theorem 2.8]{KS}. We will obtain more results in the next section.
\section{Set-valued contractions}
Kirk and Shahzad \cite{KS} also gives an analog of Theorem \ref{thmA} and Theorem \ref{thmB} for set-valued mappings. Our purpose in this section is to obtain an analog of Theorem \ref{thm2.1} and Theorem \ref{thm2.5} for set-valued mappings as well. First let us recall some definitions and results concerning set-valued mappings.\\
\indent Let $(X,d)$ be a metric space and let $\mathcal{CB}(X)$ be the family of nonempty, closed, and bounded subsets of $X$. The usual Hausdorff distance on $\mathcal{CB}(X)$ is defined as
$$
H(A,B) = \max\{\rho(A,B),\rho(B,A)\},
$$
where $A, B\in \mathcal{CB}(X)$, $\rho(A,B) = \sup_{x\in A}d(x,B)$, $\rho(B,A) = \sup_{x\in B}d(x,A)$.
\begin{defn}
Let $T:X\to\mathcal{CB}(X)$. Then
\begin{itemize}
\item[(i)] $T$ is called \emph{a multivalued contraction mapping} if there exists a constant $k\in (0,1)$ such that $H(Tx,Ty)\leq kd(x,y)$ for all $x, y\in X$.
\item[(ii)] For $\varepsilon>0$ and $k\in (0,1)$, $T$ is called \emph{an $(\varepsilon,k)$-uniform local multivalued contraction} if for every $x, y\in X$
$$
d(x,y)<\varepsilon \Rightarrow H(Tx,Ty)\leq kd(x,y).
$$
\item[(iii)] A point $x\in X$ is said to be \emph{a fixed point of $T$} if $x\in Tx$.
\end{itemize}
\end{defn} 
Kirk and Shahzad's results on set-valued mappings which will be extended are as follows:
\begin{thm}\label{kirkandshahzadthm3.4}
(Kirk and Shahzad \cite[Theorem 3.4]{KS}) Let $(X,d)$ be a metric space and $T:X\to\mathcal{CB}(X)$. Suppose there exists a metric transform $\phi$ and $k\in (0,1)$ such that the following conditions hold:
\begin{itemize}
\item[(a)] For each $x, y\in X$, $\phi(H(Tx,Ty))\leq kd(x,y)$.
\item[(b)] There exists $c\in (0,1)$ such that for $t>0$ sufficiently small, $kt\leq \phi(ct)$.
\end{itemize}
Then for $\varepsilon > 0$ sufficiently small, $T$ is an $(\varepsilon,c)$-uniform local multivalued contraction on $(X,d)$.
\end{thm}
\begin{thm}
(Kirk and Shahzad \cite[Theorem 3.6]{KS}) If, in addition to the assumptions of Theorem \ref{kirkandshahzadthm3.4}, $X$ is complete and connected, then $T$ has a fixed point.
\end{thm}
Our aim is to replace the metric transform $\phi$ in Theorem \ref{kirkandshahzadthm3.4} by a metric-preserving function. We obtain the following theorem.
\begin{thm}\label{genkirkandshahzadthm3.4}
Let $(X,d)$ be a metric space and $T:X\to\mathcal{CB}(X)$. Suppose there exists a metric-preserving function $f$ and $k\in (0,1)$ such that the following conditions hold:
\begin{itemize}
\item[(a)] For each $x, y\in X$, $f(H(Tx,Ty))\leq kd(x,y)$.
\item[(b)] $f'(0) > k$.
\end{itemize}
Then for $\varepsilon > 0$ sufficiently small, $T$ is an $(\varepsilon,c)$-uniform local multivalued contraction on $(X,d)$.
\end{thm}
\begin{cor}\label{corgenkirkandshahzadthm3.4}
With the same assumptions in Theorem \ref{genkirkandshahzadthm3.4} except that condition (b) is replaced by (b$'$): there exists $c\in (0,1)$ such that for $t>0$ sufficiently small, $kt\leq f(ct)$. Then for $\varepsilon>0$ sufficiently small, $T$ is an $(\varepsilon,c)$-uniform local multivalued contraction on $(X,d)$.
\end{cor}
\begin{thm}\label{thmaftercorgenkirkandshahzadthm3.4}
If, in addition to the assumptions of Theorem \ref{genkirkandshahzadthm3.4} or Corollary \ref{corgenkirkandshahzadthm3.4}, $X$ is complete and $\varepsilon$-chainable, then $T$ has a fixed point. In particular, if $X$ is complete and connected, then $T$ has a fixed point.
\end{thm}
The proof of these results are similar to those in Section 3.\vspace{0.3cm}\\
\noindent \textbf{Proof of Theorem \ref{genkirkandshahzadthm3.4}}\vspace{0.2cm}\\
We define $c = \frac{1}{2}\left(\frac{k}{f'(0)}+1\right)$ as in the proof of Theorem \ref{thm2.1}. Then $0 \leq \frac{k}{f'(0)} < c < 1$ and there exists $\delta_1>0$ such that for every $z\in [0,\infty)$
\begin{equation}\label{eq1genkirkandshahzadthm3.4}
0 < z\leq \delta_1 \Rightarrow \frac{f(z)}{z} > \frac{k}{c}.
\end{equation}
To show that $T$ is an $(\varepsilon,c)$-uniform local multivalued contraction for $\varepsilon>0$ sufficiently small, we let $0<\varepsilon<\frac{\delta_1}{2}$ and let $x, y\in X$ be such that $d(x,y)<\varepsilon$. By Lemma \ref{lemma10} and (\ref{eq1genkirkandshahzadthm3.4}), we have for every $b\in [0,\infty)$
\begin{equation}\label{eq2genkirkandshahzadthm3.4}
b\geq\frac{\delta_1}{2} \Rightarrow f(b)\geq \frac{f(\delta_1)}{2} > \frac{k\delta_1}{2c} > \frac{k\varepsilon}{c} > k\varepsilon. 
\end{equation}
By condition (a), we havve $f(H(Tx,Ty))\leq kd(x,y) < k\varepsilon$. Therefore we obtain by (\ref{eq2genkirkandshahzadthm3.4}) that 
\begin{equation}\label{eq3genkirkandshahzadthm3.4}
H(Tx,Ty) < \frac{\delta_1}{2}.
\end{equation}
If $d(x,y)=0$ or $H(Tx,Ty) = 0$, then it is obvious  that $H(Tx,Ty) \leq cd(x,y)$ and we are done. So assume that $H(Tx,Ty) > 0$ and $d(x,y) > 0$. Then
$$
\frac{k}{c}\frac{H(Tx,Ty)}{d(x,y)} < \frac{f(H(Tx,Ty))}{H(Tx,Ty)}\cdot\frac{H(Tx,Ty)}{d(x,y)} = \frac{f(H(Tx,Ty))}{d(x,y)} \leq k,
$$ 
where the first inequality is obtained by applying (\ref{eq3genkirkandshahzadthm3.4}) and (\ref{eq1genkirkandshahzadthm3.4}) and the last inequality is merely the condition (a). This implies $H(Tx,Ty) \leq cd(x,y)$, as desired.\qed\vspace{0.3cm}\\
\noindent\textbf{Proof of Corollary \ref{corgenkirkandshahzadthm3.4}}\vspace{0.2cm}\\
We can imitate the proof of Theorem \ref{thm2.5} to obtain $f'(0)>k$. So Corollary \ref{corgenkirkandshahzadthm3.4} follows immediately from Theorem \ref{genkirkandshahzadthm3.4}.\qed\vspace{0.3cm}\\
\noindent\textbf{Proof of Theorem \ref{thmaftercorgenkirkandshahzadthm3.4}}\vspace{0.2cm}\\
This follows from Theorem \ref{genkirkandshahzadthm3.4}, Corollary \ref{corgenkirkandshahzadthm3.4}, and Lemma \ref{nadlerthm}. The other part follows from the fact that a connected metric space is $\varepsilon$-chainable for every $\varepsilon>0$.\qed\vspace{0.3cm}\\
\noindent \textbf{Conclusion:} We replace the metric transform $\phi$ by a metric-preserving function. Therefore we obtain theorems more general than those of Kirk and Shahzad \cite[Theorem 2.2, Theorem 2.3, Theorem 2.8, Theorem 3.4, and Theorem 3.6]{KS}.
\section{Fixed point set of metric transforms and metric-preserving functions}
Recall that for a function $f:X\to X$, we denote by $\fix f$ the set of all fixed points of $f$. We begin this section with the following lemma.
\begin{lem}\label{fpmtlemma3.1}
Let $f:[0,\infty)\to [0,\infty)$ be a metric transform. If $0<a<b$, $f(a) = a$, and $f(b) = b$, then $[a,b]\subseteq \fix f$.
\end{lem}
\begin{proof}
Since $f$ is amenable and concave, the function $x \mapsto \frac{f(x)}{x}$ is decreasing on $(0,\infty)$ by Lemma \ref{amencavegotdecreas}. So if $a\leq x \leq b$, then $1 = \frac{f(a)}{a} \geq \frac{f(x)}{x} \geq \frac{f(b)}{b} = 1$, which implies $f(x) = x$. This shows that $[a,b]\subseteq \fix f$. 
\end{proof}
\begin{lem}\label{fpmtthm3.2}
If $f:[0,\infty)\to [0,\infty)$ is a metric transform, then $\fix f$ is a closed subset of $[0,\infty)$.
\end{lem}
\begin{proof}
Let $(a_n)$ be a sequence in $\fix f$ and $a_n\to a$. If $a = 0$ or $a = a_n$ for some $n\in \mathbb N$, then $a\in \fix f$ and we are done. So assume that $a > 0$ and $a \neq a_n$ for any $n\in \mathbb N$. Since $a>0$ and $a_n\to a$, $a_n>0$ for all large $n$. By passing to the subsequence, we can assume that $a_n>0$ for every $n\in \mathbb N$. It is well-known that every sequence of real numbers has a monotone subsequence (see e.g. \cite[p.\ 62]{TBB}). By passing to the subsequence again, we can assume that $(a_n)$ is monotone. Now suppose that $(a_n)$ is increasing. Then by Lemma \ref{fpmtlemma3.1}, 
$$
[a_1,a_n] \subseteq [a_1,a_2]\cup [a_2,a_3]\cup\cdots \cup [a_{n-1},a_n] \subseteq \fix f\quad\text{for every $n\in \mathbb N$}.
$$
Since $(a_n)$ is increasing and $a_n\to a$, if $a_1\leq x < a$, then there exists $N\in \mathbb N$ such that $a_1\leq x < a_N$, which implies that $x\in \fix f$, by Lemma \ref{fpmtlemma3.1}. This shows that $[a_1,a)\subseteq \fix f$. Since $f$ is increasing and $a_n<a$, $a_n = f(a_n) \leq f(a)$ for every $n\in \mathbb N$. Since $a_n \leq f(a)$ for every $n\in \mathbb N$ and $a_n\to a$, we have 
\begin{equation}\label{eq1fpmtlemma3.1}
a\leq f(a).
\end{equation}
In addition, we obtain by Lemma \ref{amencavegotdecreas} and the fact that $a\geq a_1$ that
 \begin{equation}\label{eq2fpmtlemma3.1}
\frac{f(a)}{a} \leq \frac{f(a_1)}{a_1} = 1.
\end{equation}
From (\ref{eq1fpmtlemma3.1}) and (\ref{eq2fpmtlemma3.1}), we obtain $f(a) = a$, as required. The case where $(a_n)$ is decreasing can be proved similarly. This completes the proof.
\end{proof}
\begin{lem}\label{lemma28.5}
Let $f:[0,\infty)\to[0,\infty)$ be a metric transform. Then $\fix f = [0,\infty)$ if and only if $\sup\fix f = +\infty$.
\end{lem}
\begin{proof}
It is enough to show that $\sup\fix f = +\infty$ implies $(0,\infty)\subseteq \fix f$. So suppose that $\sup\fix f = +\infty$ but there exists $x\in (0,\infty)$ such that $f(x)\neq x$. Since $\sup\fix f = +\infty$, there exists $a>x$ such that $f(a) = a$. Similarly, there exists $b>a$ such that $f(b) = b$. Since $f$ is amenable and concave, we obtain by Lemma \ref{amencavegotdecreas} that 
$$
\frac{f(x)}{x} \geq \frac{f(a)}{a} = 1.
$$
Since $f(x)\neq x$, $f(x) > x$. Since $x<a<b$, there exists $t\in (0,1)$ such that $a = (1-t)x+tb$. By the concavity of $f$, we obtain
$$
a = f(a) = f\left((1-t)x+tb\right) \geq (1-t)f(x)+tf(b) > (1-t)x+tb = a,
$$
a contradiction. This completes the proof.
\end{proof}
\begin{thm}\label{chapter5thm38}
If $a>0$, then each set of the form $\{0\}$, $\{0,a\}$, $[0,a]$, and $[0,\infty)$ is a fixed point set of a metric transform. Conversely, if $f$ is a metric transform, then $\fix f = \{0\}, \{0,a\}, [0,a]$, or $[0,\infty)$ for some $a\in (0,\infty)$.
\end{thm}
\begin{proof}
Define $f_1, f_2, f_3, f_4 :[0,\infty)\to [0,\infty)$ by 
$$
f_1(x) = \frac{x}{2},\quad f_2(x) = \sqrt{ax},\quad f_3(x) = x,\quad f_4(x) = \begin{cases}
x,\quad &x\in[0,a];\\
\frac{x+a}{2},\quad& x>a.
\end{cases}
$$ 
It is easy to verify that the functions $f_1, f_2, f_3, f_4$ are metric transforms and $\fix f_1 = \{0\}$, $\fix f_2 = \{0,a\}$, $\fix f_3 = [0,\infty)$, and $\fix f_4 = [0,a]$. This proves the first part. \\
\indent Next let $f$ be a metric transform such that $\fix f\neq \{0\}$ and $\fix f \neq [0,\infty)$. We let $a = \sup\fix f$ and assert that $\fix f = \{0, a\}$ or $[0,a]$. Note that since $\fix f\neq \{0\}$, $a>0$. It is obtained by Lemma \ref{lemma28.5} that $a<+\infty$. Now apply Lemma \ref{fpmtthm3.2} to get $a\in \fix f$. Therefore $\{0,a\}\subseteq \fix f$. By the definition of $a$, we see that $x\notin \fix f$ for every $x>a$. Now if $x\notin \fix f$ for every $0<x<a$, then $\fix f = \{0,a\}$ and we are done. So assume that there exists $0<x<a$ such that $x\in \fix f$. We will show that $\fix f = [0,a]$. Since $a  = \sup \fix f$, it is obvious that $\fix f\subseteq [0,a]$. Suppose for a contradiction that there exists $0<y<a$ such that $f(y)\neq y$. Since $0<x<a$ and $x, a\in \fix f$, we obtain by Lemma \ref{fpmtlemma3.1} that $y\notin [x,a]$. So $y<x$. By Lemma \ref{amencavegotdecreas} we have 
$$
\frac{f(y)}{y} \geq \frac{f(x)}{x} = 1.
$$
Since $f(y)\neq y$, $f(y)>y$. Since $y<x<a$, there exists $t\in (0,1)$ such that $x = (1-t)y+ta$. By the concavity of $f$, we obtain 
$$
x = f(x) = f((1-t)y+ta) \geq (1-t)f(y)+tf(a) > (1-t)y+ta = x,
$$
a contradiction. This completes the proof.
\end{proof} 
Since every metric transform is metric-preserving, we immediately obtain that each set of the form $\{0\}$, $\{0,a\}$, $[0,a]$, and $[0,\infty)$ is a fixed point set of a metric-preserving function. However, there is a metric-preserving function $f$ where $\fix f$ is not of this form. Let us show this more precisely.
\begin{cor}
If $a>0$, then each set of the form $\{0\}$, $\{0,a\}$, $[0,a]$, and $[0,\infty)$ is a fixed point of a metric-preserving function.
\end{cor} 
\begin{proof}
This follows immediately from Theorem \ref{chapter5thm38} and Proposition \ref{prop2.3}.
\end{proof}
\begin{exam}\label{addmoreexam40}
Let $f, g, h:[0,\infty)\to[0,\infty)$ be given by
\begin{align*}
f(x) &= \left\lceil x\right\rceil,\quad\quad g(x) =
 \begin{cases}
0,\quad &x=0;\\
1,\quad &x\in \mathbb Q-\{0\};\\
\sqrt 2,\quad &x\in \mathbb Q^c,
\end{cases} \\
h(x) &= \begin{cases}
0,\quad &x=0;\\
1,\quad &0<x<1;\\
x,\quad &x\in\mathbb Q\cap [1,2];\\
2,\quad &x\in \left(\mathbb Q^c\cap [1,2]\right) \cup (2,\infty). 
\end{cases}
\end{align*}
(Recall that $\left\lceil x\right\rceil$ is the smallest integer which is larger or equal to $x$) It is easy to verify that $f$ is amenable, increasing, and subadditive. So by Lemma \ref{chapter2lemma9}, $f$ is metric-preserving. Since $g$ and $h$ are amenable and tightly bounded, we obtain by Lemma \ref{lemma8} that $g$ and $h$ are metric-preserving. It is easy to see that $\fix f = \mathbb N\cup \{0\}$, $\fix g = \{0, 1, \sqrt 2\}$, and $\fix h = \{0\}\cup (\mathbb Q\cap [1,2])$.
\end{exam}
By generating a function similar to $h$ we obtain a more general result as follows:
\begin{prop}\label{addmoreprop41}
Let $A \subseteq [u,2u]$ for some $u>0$. Then $A\cup\{0\}$ is a fixed point set of a metric-preserving function. 
\end{prop}
\begin{proof}
We define $f:[0,\infty)\to[0,\infty)$ by
$$
f(x) = \begin{cases}
0,\quad &\text{if $x=0$};\\
x,\quad &\text{if $x\in A$};\\
u,\quad &\text{if $x\notin A \wedge x\notin \{0,u\}$},
\end{cases}
$$
and if $u\notin A$, then define $f(u) = 2u$. Then $f$ is amenable and tightly bounded. Therefore, by Lemma \ref{lemma8}, $f$ is metric-preserving. It is easy to see that $\fix f = A\cup\{0\}$. This completes the proof.
\end{proof}
From Example \ref{addmoreexam40} and Proposition \ref{addmoreprop41}, we see that the fixed point set of a metric-preserving function may not be of the form $\{0\}$, $\{0,a\}$, $[0,a]$, and $[0,\infty)$. Other natural questions and answers are the following:
\begin{itemize}
\item[Q1:] Is there a metric-preserving function which does not satisfy the result in Lemma \ref{fpmtlemma3.1}?
\item[A1:] Every function given in Example \ref{addmoreexam40} is such a function.
\item[Q2:] Is there a metric-preserving function which does not satisfy the result in Lemma \ref{fpmtthm3.2}?
\item[A2:] The function $h$ given in Example \ref{addmoreexam40} and the function $f$ given in Proposition \ref{addmoreprop41} (with a suitable set $A$) are such functions.
\item[Q3:] Is there a metric-preserving function which does not satisfy the result in Lemma \ref{lemma28.5}?
\item[A3:] The function $f$ given in Example \ref{addmoreexam40} is such a function.
\end{itemize} 
We see that the fixed point sets of metric-preserving functions are quite difficult to be completely characterized. We leave this to the interested reader. Now we end this article by giving continuous metric-preserving functions which do not satisfy the results in Lemma \ref{fpmtlemma3.1} and Lemma \ref{lemma28.5}.
\begin{exam}
Let $f, g:[0,\infty)\to[0,\infty)$ be given by $f(x) = \left\lfloor x\right\rfloor+\sqrt{x-\left\lfloor x\right\rfloor}$ and $g(x) = x+|\sin x|$. (Recall that $\left\lfloor x\right\rfloor$ is the largest integer which is less than or equal to $x$). We will use Lemma \ref{Pokornylemma} to show that $f$ and $g$ are metric-preserving. First, the function $x\mapsto |\sin x|$ is periodic with period $\pi$.
$$
|\sin(x+y)| = |\sin x\cos y+\cos x\sin y| \leq |\sin x|+|\sin y|.
$$
So the function $x\mapsto |\sin x|$ is also subadditive. From this, we easily see that $g$ satisfies the condition in Lemma \ref{Pokornylemma}. So $g$ is metric-preserving. It is not difficult to verify that $f$ is also satisfies the assumption in Lemma \ref{Pokornylemma} and we will leave the details to the reader. It is also easy to see that $\fix f = \mathbb N\cup \{0\}$ and $\fix g = \{n\pi\mid n\in \mathbb N\cup\{0\}\}$. So $f$ and $g$ are continuous metric preserving functions of which fixed point sets do not satisfy the results in Lemma \ref{fpmtlemma3.1} and Lemma \ref{lemma28.5}. 
\end{exam} 
\noindent \textbf{Competing Interests}\\[0.2cm]
The authors declare that they have no competing interests.\\[0.3cm]
\noindent \textbf{Authors' contributions}\\[0.2cm]
All authors contributed significantly in writing this paper. All authors read and approved this final manuscript.\\[0.3cm]
\noindent \textbf{Author details}\\[0.2cm]
$^a$Department of Mathematics, Faculty of Science, Silpakorn University, Nakhon Pathom, 73000, Thailand, $^b$Department of Mathematics and Computer Science, Faculty of Science, Chulalongkorn University, Bangkok, 10330, Thailand.\\[0.3cm]
\noindent\textbf{Acknowledgment} 
The first author would like to thank The Thailand Research Fund for financial support under the contract number TRG5680052. Both authors also would like to thank Professor I. A. Rus for sending them his article on the role of metric-preserving functions in fixed point theory which helps improve the presentation of this article.


\begin{thebibliography}{99}
\bibitem {B1} Blumenthal, LM: Theory and Applications of Distance Geometry, 2nd edn. Chelsea, New York (1970) 
\bibitem {B2} Blumenthal, LM: Remarks concerning the Euclidean four--point property. Ergebnisse Math. Kolloq. Wein \textbf{7}, 7--10 (1936)
\bibitem {BD} Bors\'ik, J, Dobo\v s, J: On metric preserving functions. Real Anal. Exchange. \textbf{13}, 285--293 (1987--88)
\bibitem {BD1} Bors\'ik, J, Dobo\v s, J: Functions whose composition with every metric is a metric. Math. Slovaca. \textbf{31}, 3--12 (1981)
\bibitem {C} Corazza, P: Introduction to metric-preserving functions. Amer. Math. Monthly. \textbf{106}(4), 309--323 (1999)
\bibitem {Da} Das, PP: Metricity preserving transforms. Pattern Recognition Letters. \textbf{10}, 73--76 (1989) 
\bibitem {Do2} Dobo\v s, J: Metric Preserving Functions, Online Lecture Notes available at \\
http://web.science.upjs.sk/jozefdobos/wp-content/uploads/2012/03/mpf1.pdf
\bibitem {Do} Dobo\v s, J: On modification of the Euclidean metric on reals. Tatra Mt. Math. Publ. \textbf{8}, 51--54 (1996) 
\bibitem {Do1} Dobo\v s, J: A survey of metric-preserving functions. Questions Answers Gen. Topology. \textbf{13}, 129--133 (1995)
\bibitem {DP} Dobo\v s, J, Piotrowski, Z: When distance means money. Internat. J. Math. Ed. Sci. Tech. \textbf{28}, 513--518 (1997)  
\bibitem {DP1} Dobo\v s, J, Piotrowski, Z: A note on metric-preserving functions. Internat. J. Math. Math. Sci. \textbf{19}, 199--200 (1996)
\bibitem {DP2} Dobo\v s, J, Piotrowski, Z: Some remarks on metric-preserving functions. Real Anal. Exchange. \textbf{19}, 317--320 (1993--94)
\bibitem {H} Holmes, RD: Fixed points for local radial contractions. In: Swaminathan, S(ed.) Fixed Point Theory and Its Applications, pp. 79--89. Academic Press, New York (1976)  
\bibitem {HK} Hu, T, Kirk, WA: Local contractions in metric spaces. Proc. Am. Math. Soc. \textbf{68}, 121--124 (1978)
\bibitem {J} Jungck, G: Local radial contractions--a counter--example. Houst. J. Math. \textbf{8}, 501--506 (1982)
\bibitem {KS} Kirk, WA, Shahzad, N: Remarks on metric transforms and fixed-point theorems. Fixed Point Theory and Applications 2013, \textbf{2013}:106.
\bibitem {Na} Nadler, SB Jr: Multi-valued contraction mappings. Pac. J. Math. \textbf{30}, 475--488 (1969)
\bibitem {PRS} Petru\c{s}el, A, Rus, IA, \c{S}erban, MA: The role of equivalent metrics in fixed point theory. Topol. Methods Nonlinear Anal. \textbf{41}(1), 85--112 (2013)
\bibitem {PV} Piotrowski, Z, Vallin, RW: Functions which preserve Lebesgue spaces. Comment. Math. Prace Mat. \textbf{43}(2), 249--255 (2003)
\bibitem {Po} Pokorn\'y, I: Some remarks on metric-preserving functions. Tatra Mt. Math. Publ. \textbf{2}, 65--68 (1993)
\bibitem {Po1} Pokorn\'y, I: Some remarks on metric-preserving functions of several variables. Tatra Mt. Math. Publ. \textbf{8}, 89--92 (1996) 
\bibitem {PT1} Pongsriiam, P, Termwuttipong, I: Remarks on ultrametrics and metric-preserving functions, preprint
\bibitem {Sr} Sreenivasan, TK: Some properties of distance functions. J. Indian. Math. Soc. (N.S.) \textbf{11}, 38--43 (1947)
\bibitem {T} Tan, KK: Fixed point theorems for nonexpansive mappings. Pac. J. Math. \textbf{41}, 829--842 (1972)
\bibitem {TO} Termwuttipong, I, Oudkam, P: Total boundedness, completeness and uniform limits of metric-preserving functions. Ital. J. Pure Appl. Math. \textbf{18}, 187--196 (2005)
\bibitem {TBB} Thomson, BS, Bruckner, JB, Bruckner, AM: Elementary Real Analysis, Prentice Hall, New Jersey (2001)
\bibitem {V} Vallin, RW: Continuity and differentiability aspects of metric preserving functions. Real Anal. Exchange. \textbf{25}(2), 849--868 (1999/00)
\bibitem {Wi} Wilson, WA: On certain types of continuous transformations of metric spaces. Amer. J. Math. \textbf{57}, 62--68 (1935) 
\end{thebibliography}
\end{document}